\documentclass[12pt]{amsart}
\usepackage{amsmath,amssymb,amsbsy,amsfonts,latexsym,amsopn,amstext,cite,
                                               amsxtra,euscript,amscd,bm}
\usepackage{url}

\usepackage{mathrsfs}

\usepackage{color}
\usepackage[colorlinks,linkcolor=blue,anchorcolor=blue,citecolor=blue,backref=page]{hyperref}
\usepackage{color}
\usepackage{graphics,epsfig}
\usepackage{graphicx}
\usepackage{float}
\usepackage{epstopdf}
\hypersetup{breaklinks=true}

\usepackage[np]{numprint}
\npdecimalsign{\ensuremath{.}}

\def\Th1{\varTheta}

\usepackage{bibentry}

\usepackage[english]{babel}
\usepackage{mathtools}
\usepackage{todonotes}
\usepackage{url}
\usepackage[colorlinks,linkcolor=blue,anchorcolor=blue,citecolor=blue,backref=page]{hyperref}

\usepackage[norefs,nocites]{refcheck}

\begin{document}

\newtheorem{theorem}{Theorem}
\newtheorem{lemma}[theorem]{Lemma}
\newtheorem{claim}[theorem]{Claim}
\newtheorem{cor}[theorem]{Corollary}
\newtheorem{conj}[theorem]{Conjecture}
\newtheorem{prop}[theorem]{Proposition}
\newtheorem{definition}[theorem]{Definition}
\newtheorem{question}[theorem]{Question}
\newtheorem{example}[theorem]{Example}
\newcommand{\hh}{{{\mathrm h}}}
\newtheorem{remark}[theorem]{Remark}

\numberwithin{equation}{section}
\numberwithin{theorem}{section}
\numberwithin{table}{section}
\numberwithin{figure}{section}

\def\sssum{\mathop{\sum\!\sum\!\sum}}
\def\ssum{\mathop{\sum\ldots \sum}}
\def\iint{\mathop{\int\ldots \int}}

\newcommand{\diam}{\operatorname{diam}}

\def\squareforqed{\hbox{\rlap{$\sqcap$}$\sqcup$}}
\def\qed{\ifmmode\squareforqed\else{\unskip\nobreak\hfil
\penalty50\hskip1em \nobreak\hfil\squareforqed
\parfillskip=0pt\finalhyphendemerits=0\endgraf}\fi}

\newfont{\teneufm}{eufm10}
\newfont{\seveneufm}{eufm7}
\newfont{\fiveeufm}{eufm5}
%
%
\newfam\eufmfam
     \textfont\eufmfam=\teneufm
\scriptfont\eufmfam=\seveneufm
     \scriptscriptfont\eufmfam=\fiveeufm
%
%
\def\frak#1{{\fam\eufmfam\relax#1}}

\newcommand{\bflambda}{{\boldsymbol{\lambda}}}
\newcommand{\bfmu}{{\boldsymbol{\mu}}}
\newcommand{\bfxi}{{\boldsymbol{\eta}}}
\newcommand{\bfrho}{{\boldsymbol{\rho}}}

\def\eps{\varepsilon}

\def\fK{\mathfrak K}
\def\fT{\mathfrak{T}}
\def\fL{\mathfrak L}
\def\fR{\mathfrak R}

\def\fA{{\mathfrak A}}
\def\fB{{\mathfrak B}}
\def\fC{{\mathfrak C}}
\def\fM{{\mathfrak M}}
\def\fS{{\mathfrak  S}}
\def\fU{{\mathfrak U}}
\def\fW{{\mathfrak W}}

\def\T {\mathsf {T}}
\def\Tor{\mathsf{T}_d}
\def\Tore{\widetilde{\mathrm{T}}_{d} }

\def\sM {\mathsf {M}}

\def\ss{\mathsf {s}}

\def\Kmnd{\cK_d(m,n)}
\def\Kmnp{\cK_p(m,n)}
\def\Kmnq{\cK_q(m,n)}

\def \balpha{\bm{\alpha}}
\def \bbeta{\bm{\beta}}
\def \bgamma{\bm{\gamma}}
\def \bdelta{\bm{\delta}}
\def \bzeta{\bm{\zeta}}
\def \blambda{\bm{\lambda}}
\def \bchi{\bm{\chi}}
\def \bphi{\bm{\varphi}}
\def \bpsi{\bm{\psi}}
\def \bnu{\bm{\nu}}
\def \bomega{\bm{\omega}}

\def \bell{\bm{\ell}}

\def\eqref#1{(\ref{#1})}

\def\vec#1{\mathbf{#1}}

\newcommand{\abs}[1]{\left| #1 \right|}

\def\Zq{\mathbb{Z}_q}
\def\Zqx{\mathbb{Z}_q^*}
\def\Zd{\mathbb{Z}_d}
\def\Zdx{\mathbb{Z}_d^*}
\def\Zf{\mathbb{Z}_f}
\def\Zfx{\mathbb{Z}_f^*}
\def\Zp{\mathbb{Z}_p}
\def\Zpx{\mathbb{Z}_p^*}
\def\cM{\mathcal M}
\def\cE{\mathcal E}
\def\cH{\mathcal H}

\def\le{\leqslant}

\def\ge{\geqslant}

\def\sfB{\mathsf {B}}
\def\sfC{\mathsf {C}}
\def\L{\mathsf {L}}
\def\FF{\mathsf {F}}

\def\sE {\mathscr{E}}
\def\sS {\mathscr{S}}

\def\cA{{\mathcal A}}
\def\cB{{\mathcal B}}
\def\cC{{\mathcal C}}
\def\cD{{\mathcal D}}
\def\cE{{\mathcal E}}
\def\cF{{\mathcal F}}
\def\cG{{\mathcal G}}
\def\cH{{\mathcal H}}
\def\cI{{\mathcal I}}
\def\cJ{{\mathcal J}}
\def\cK{{\mathcal K}}
\def\cL{{\mathcal L}}
\def\cM{{\mathcal M}}
\def\cN{{\mathcal N}}
\def\cO{{\mathcal O}}
\def\cP{{\mathcal P}}
\def\cQ{{\mathcal Q}}
\def\cR{{\mathcal R}}
\def\cS{{\mathcal S}}
\def\cT{{\mathcal T}}
\def\cU{{\mathcal U}}
\def\cV{{\mathcal V}}
\def\cW{{\mathcal W}}
\def\cX{{\mathcal X}}
\def\cY{{\mathcal Y}}
\def\cZ{{\mathcal Z}}
\newcommand{\rmod}[1]{\: \mbox{mod} \: #1}

\def\cg{{\mathcal g}}

\def\vy{\mathbf y}
\def\vr{\mathbf r}
\def\vx{\mathbf x}
\def\va{\mathbf a}
\def\vb{\mathbf b}
\def\vc{\mathbf c}
\def\ve{\mathbf e}
\def\vh{\mathbf h}
\def\vk{\mathbf k}
\def\vm{\mathbf m}
\def\vz{\mathbf z}
\def\vu{\mathbf u}
\def\vv{\mathbf v}

\def\e{{\mathbf{\,e}}}
\def\ep{{\mathbf{\,e}}_p}
\def\eq{{\mathbf{\,e}}_q}

\def\Tr{{\mathrm{Tr}}}
\def\Nm{{\mathrm{Nm}}}

 \def\SS{{\mathbf{S}}}

\def\lcm{{\mathrm{lcm}}}

 \def\0{{\mathbf{0}}}

\def\({\left(}
\def\){\right)}
\def\l|{\left|}
\def\r|{\right|}
\def\fl#1{\left\lfloor#1\right\rfloor}
\def\rf#1{\left\lceil#1\right\rceil}
\def\fl#1{\left\lfloor#1\right\rfloor}
\def\ni#1{\left\lfloor#1\right\rceil}
\def\sumstar#1{\mathop{\sum\vphantom|^{\!\!*}\,}_{#1}}

\def\mand{\qquad \mbox{and} \qquad}

\def\tblue#1{\begin{color}{blue}{{#1}}\end{color}}




\hyphenation{re-pub-lished}

\mathsurround=1pt

\def\bfdefault{b}

\def \F{{\mathbb F}}
\def \K{{\mathbb K}}
\def \N{{\mathbb N}}
\def \Z{{\mathbb Z}}
\def \P{{\mathbb P}}
\def \Q{{\mathbb Q}}
\def \R{{\mathbb R}}
\def \C{{\mathbb C}}
\def\Fp{\F_p}
\def \fp{\Fp^*}

 \def \xbar{\overline x}

\title[Sums of Kloosterman sums]
{Sums of Kloosterman sums over square-free and smooth 
integers} 

\author [X. Shao]{Xuancheng Shao}\thanks{XS was supported by NSF Grant DMS-2200565.}
\address{Department of Mathematics, University of Kentucky, 715 Patterson Office Tower, Lexington, KY 40506, USA}
\email{xuancheng.shao@uky.edu}

\author[I. E. Shparlinski] {Igor E. Shparlinski}\thanks{I.S. was supported in part by ARC Grant~DP230100534}
\address{School of Mathematics and Statistics, University of New South Wales, Sydney NSW 2052, Australia}
\email{igor.shparlinski@unsw.edu.au}

\author [L. P. Wijaya]{Laurence P. Wijaya}
\address{Department of Mathematics, University of Kentucky, 715 Patterson Office Tower, Lexington, KY 40506, USA}
\email{laurence.wijaya@uky.edu}

\begin{abstract}  Recently there has been a large number of works on bilinear sums with Kloosterman sums
and on sums of Kloosterman sums twisted by arithmetic functions. Motivated by these, we consider 
several related new questions about sums of Kloosterman sums parametrised by square-free and smooth
integers. 
 \end{abstract}


\maketitle

\tableofcontents

\section{Introduction}
\subsection{Motivation}
For a prime $p$ and an integer $n$, we define the $s$-dimensional Kloosterman sum
$$ \cK_{s, p}(n) = p^{-(s-1)/2} \sum_{\substack{x_1, \ldots, x_s=1\\
x_1\cdots x_s \equiv n \pmod p}}^{p-1} \ep\(x_1 +\cdots + x_s\)
$$
where $\ep(x) = \exp(2\pi ix/p)$. The celebrated result of Deligne~\cite{Del}
gives the following bound:
 \begin{equation}
\label{eq:Deligne}
 |\cK_{s, p}(n)| \le s.
\end{equation}
See also~\cite[Section~11.11]{IwKow}. In the classical case of $s=2$ we denote 
$$
 \cK_{p}(n) =  \cK_{2, p}(n).
 $$

Recently,  there  has been active interest in estimating sums of Kloosterman sums either over sequences of parameters $n$ of arithmetic interest, or twisted 
by arithmetic functions, such as, 
\begin{equation}
        \label{eq:Sums}
P_{s, p}(L) = \sum_{\substack{\ell \le L\\ \ell~\text{prime}}}
\cK_{s, p}(\ell) \quad \text{and} \quad M_{s, p}(f;N) =  \sum_{n \le N} f(n)\cK_{s, p}(n)
\end{equation}
with some multiplicative function $f(n)$ such as the M\"obius function
$f(n) = \mu(n)$ or the divisor function $f(n) = \tau(n)$.
See, for example,~\cite{BFKMM2, FKM1, KoSh, LSZ}.

These results largely rely on recent progress on bounds of bilinear Type~I and Type~II sums with Kloosterman sums; 
see~\cite{BagShp, BFKMM1, BFKMM2, FKM1, KSWX,KMS1,KMS2, Shp, ShpZha}.
In particular,~\cite[Theorem~1.5]{FKM1} implies power-saving bounds on the 
sums $P_{s, p}(L)$ given by~\eqref{eq:Sums} with $L \ge p^{3/4+\varepsilon}$ for an arbitrary fixed $\varepsilon > 0$. These bounds were subsequently improved in~\cite[Theorem~1.8]{BFKMM2} in the case of $s=2$ and for the same range of $L$.

For the sums $M_{s, p}(f; N)$ given by~\eqref{eq:Sums},~\cite[Theorem~1.7]{FKM1} gives power-saving bounds on $M_{s, p}(\mu; N)$ when $f(n) = \mu(n)$, provided that 
$N\ge p^{3/4+\varepsilon}$ for an arbitrary fixed $\varepsilon > 0$. Similar bounds are given by~\cite[Corollary~1.4]{KMS2} on $M_{s, p}(\tau; N)$ when $f(n) = \tau(n)$, provided that $N\ge p^{2/3+\varepsilon}$.
These  thresholds have both been reduced to $N\ge p^{1/2+\varepsilon}$
in~\cite{KoSh}, however with a logarithmic saving instead of a power saving in the bound.

\subsection{Sums of Kloosterman sums over square-free numbers}

In this paper we first consider sums of Kloosterman sums over square-free numbers (that is, over numbers which are not divisible by the square of a prime):
$$
Q_{s, p}(N)  =  \sum_{n \le N} |\mu(n)| \cK_{s, p}(n) 
= \sum_{\substack{n \le N\\ n~\text{square-free}}}  \cK_{s, p}(n).
$$

Note that we have the trivial bound 
 \begin{equation}
\label{eq:Triv 1}
Q_{s, p}(N)  \ll  N
\end{equation}
implied  by~\eqref{eq:Deligne}. Our goal is twofold: to obtain a nontrivial bound (possibly with a power saving) for $N$ as small as possible and to have a bound as good as possible for any given $N$.

Our main result in this direction is the following bound.

\begin{theorem}
\label{thm:SquareFree any s}
    For any positive integer $s \ge 2$ and even positive integer $\ell$,  if  $p^{1/2+2/\ell} \le N \le p$ then 
$$
|Q_{s, p}(N)| \le  N^{1/2} p^{1/4} \(\frac{p^{1/2+2/\ell}}{N}\)^{\tfrac{1}{2(4\ell-3)}} p^{o(1)}.
$$   
\end{theorem}

\begin{remark} Our approach works for $N >  p$ 
as well. However, the optimisation 
of our argument becomes more cluttered in such generality. Since we are mostly interested in short sums and in 
order to exhibit  our ideas in tje simplest possible form, we focus on the case $N \le   p$. 
\end{remark}

Theorem~\ref{thm:SquareFree any s} is proved in Section~\ref{sec:SquareFree}. It in particular implies that if $p^{1/2+\varepsilon} \le N \le p$ for an arbitrary fixed $\varepsilon$, then 
 \begin{equation}
\label{eq:Triv 2}
Q_{s, p}(N)  \ll  N^{1/2} p^{1/4+o(1)}.
\end{equation}
In fact, the bound~\eqref{eq:Triv 2} can be quickly obtained by combining~\eqref{eq: T-PV} and~\eqref{eq: T-Deligne} in our proof of Theorem~\ref{thm:SquareFree any s}. Compared to~\eqref{eq:Triv 2}, we are able to obtain additional savings in Theorem~\ref{thm:SquareFree any s} by studying a certain type-I sum in Lemma~\ref{lem: Type I any s}. On the other hand, both the bound in Theorem~\ref{thm:SquareFree any s} and~\eqref{eq:Triv 2} beat the trivial bound~\eqref{eq:Triv 1} in the range $p^{1/2+\eps} \le N \le p$.

For a specific value of $N$, one can optimize the bound in Theorem~\ref{thm:SquareFree any s} by making the choice of $\ell$ that minimizes the term 
$$ \(\frac{p^{1/2+2/\ell}}{N}\)^{\tfrac{1}{2(4\ell-3)}}. $$
For example, if $N=p$ then by choosing $\ell = 8$ we get
$$ Q_{s,p}(p) \ll p^{3/4-1/232+o(1)}. $$

\subsection{Sums of Kloosterman sums over smooth numbers}

Next, we recall that an integer $n$ is called $y$-smooth if  $P(n) \le y$, 
where $P(n)$ denotes the largest prime divisor of an integer $n\ge 1$. See~\cite{Granv, HilTen} for backgrounds and classical estimates on smooth numbers.
For $N \ge y \ge 2$, we denote by $\cS(N,y)$ the set of $y$-smooth positive integers  $n \le N$
and, as usual, we denote $\Psi(N,y) = \# \cS(N,y)$.

We consider the sum
$$
R_{s.p}(N,y) = 
\sum_{n \in \cS(N,y)}  \cK_{s, p}(n) ,
$$
for which we have, in analogy to~\eqref{eq:Triv 1}, the trivial bound  
$$
R_{s, p}(N,y)  \ll \Psi(N,y) \le N.
$$

As in the case of squarefree numbers, our goal is to obtain nontrivial bounds for $N$ as small as possible and, at the same time, to have a bound as good as possible for any given $N$ and $y$. Additionally, we are also interested in obtaining nontrivial bounds for $y$ as small as possible.

Before formulating our main result in this direction,  we recall some well-known estimates  in the theory of smooth numbers. Let $\alpha(N,y)$ be the {\it saddle point\/} corresponding to the $y$-smooth numbers
up to $N$ as discussed in~\cite{Harp, HilTen}. 
In particular,  $\alpha(N,y)$ satisfies 
$$
\alpha(N,y) = (1+o(1))\frac{\log(1+y/\log N)}{\log y}
$$
provided that $y \leq N$ and $y \to \infty$;
see~\cite[Theorem~2]{HilTen}.
In particular, if $\log y/\log \log N \to \infty$ then $\alpha(N,y)  \to 1$, and if $y = (\log N)^K$ for some $K \geq 1$ then $\alpha(N,y) = 1-1/K + o(1)$. We have
\begin{equation}
\label{eq: Psi N-alpha}
 \Psi(N,y)  = N^{\alpha(N,y) + o(1)};
\end{equation}
see, for example,~\cite[Section~2]{Harp} or~\cite[Theorem~1]{HilTen}.

\begin{theorem}\label{thm:smooth-Kloosterman}
    Let $p$ be prime, $N \geq p^{1/2}$, and $y \geq \log N$. Then
$$ |R_{s,p}(N,y)| \le  \Psi(N,y) y^{1/2} p^{\beta} N^{-\gamma + o(1)} , 
$$
where 
$$
 \beta = \frac{1}{4\(1+\alpha(N,y)\)} \mand  \gamma= \frac{\alpha(N,y)^2}{2\(1+\alpha(N,y)\)}. 
$$
\end{theorem}  

Some remarks on the bound are in order. If $y = N^{o(1)}$ then $y^{1/2}$ can be dropped from the bound and we have
$$R_{s,p}(N,y) \ll  \Psi(N,y) p^{\beta} N^{-\gamma+o(1)}, $$
which is non-trivial when $N \ge  p^{\beta/\gamma+ \eps} = p^{1/(2\alpha(N,y)^2) + \eps}$ for arbitrary fixed $\eps > 0$. 
In the special case $N=p$, we have a nontrivial bound for $R_{s,p}(p,y)$, provided that $\alpha(N,y) > 1/\sqrt{2}$, or $y \geq (\log N)^{2+\sqrt{2}+\eps}$.

On the other hand, if $\log y/\log\log N\rightarrow \infty$ then $\alpha(N,y) = 1 + o(1)$, and thus
Theorem~\ref{thm:smooth-Kloosterman}  yields
$$R_{s,p}(N,y) \ll   \Psi(N,y) y^{1/2}  p^{1/8} N^{-1/4+o(1)}, $$
which is non-trivial when   $N > y^2 p^{1/2+\eps}$.

The rest of the paper is organized as follows. In Section~\ref{sec:typeI} we state some preliminary lemmas including a type-I estimate which is curicial in our treatment of the square-free numbers. Then in Sections~\ref{sec:SquareFree} and~\ref{sec:smooth}, we prove Theorems~\ref{thm:SquareFree any s} and~\ref{thm:smooth-Kloosterman}, respectively.

\section{Preparations} 
\subsection{Notation}

We use the standard notations $U\ll V$ and $V \gg U$  as  equivalent to the statement $|U|\leq c V$, for some  constant $c> 0$, which throughout
this paper may depend only on the integer parameters $\ell$ and $s$.

For a finite set $\cS$ we use $\# \cS$ to denote its cardinality.

The variables of summation $d$, $k$, $m$ and $n$ are always positive integers.

We also follow the convention that fractions of the shape $1/ab$ mean $1/(ab)$ (that than 
$b/a$ as their formal interpretation requires). 

The letter $p$ always denotes a prime number, and we use $\F_p$ 
to denote the  finite field of $p$ elements. 

We denote by $p(\ell)$ and $P(\ell)$, the smallest and the largest prime factors of an integer $\ell\ne 0$, respectively.
We adopt the convention that $p(1) = P(1) = +\infty$.

Finally, we write $\sum_{k\le K}$ to denote the summation over positive integer $k \le K$.

\subsection{Preliminary bounds on sums of Kloosterman sums}\label{sec:typeI}

The  following result is a special case of~\cite[Corollary~1.6]{FKM2}.

\begin{lemma}
\label{lem:FKM}
Uniformly over   $h \in \Z$,  and 
$d \in \F_p\setminus\{0,1\}$,  we have 
\begin{align*}
& \sum_{x \in\F_p}  \cK_{s, p}\(x\) \ep(hx)\ll p^{1/2},  \\
& \sum_{x \in\F_p} \cK_{s, p}\(x\)  \cK_{s, p}\(d x\) \ep(hx) \ll p^{1/2} .
\end{align*}
\end{lemma}

The standard completion technique (see~\cite[Section~12.2]{IwKow})
shows that  Lemma~\ref{lem:FKM} implies the following bound 
on incomplete sums. 

\begin{lemma}
\label{lem:FKM Incompl}
For any $K \le p$, uniformly over   $d, e \in \F_p$ satisfying $d \ne 0$ and $e/d \ne 0, 1$, we have 
\begin{align*}
& \sum_{n \le K}  \cK_{s, p}(dn) \ll  p^{1/2} \log p,\\
& \sum_{n \le K}  \cK_{s, p}(dn)  \cK_{s, p}(en) \ll  p^{1/2} \log p.
\end{align*}
\end{lemma} 

We now record a bound  on a variant of Type-I sums of Kloosterman sums.
Namely, instead of Type-I sums with  $\cK_{s,p}(mn)$
we consider sums with $\cK_{s,p}(m^{r} n)$ where $r$ is an 
arbitrary integer (if $r$ is negative we consider the argument 
of the corresponding Kloosterman sum modulo $p$). This is obtained by a slight extension of the argument of~\cite{KMS2}. 
See also~\cite{BaSh,BFKMM1,FKM1,KSWX,KMS1,KMS2}  for  several 
other bounds on related sums.

\begin{lemma}
\label{lem: Type I any s}
Fix an integer $r \neq 0$ and an even integer $\ell \ge 2$.
Let $D, N \le p$ be positive integers with $N > 2p^{1/\ell}$. For each $d \le D$ let $\cN_d \subset [1,N]$ be an interval. Then for any complex weights
$$
        \bm{\alpha}=\{\alpha_d\}_{d \le D}
$$
with $\alpha_d \ll 1$,  we have
$$
\sum_{d\le D} \sum_{n\in \cN_d} \alpha_d \cK_{s,p}(d^r n) 
\ll DN\(N^{-1}  +  \frac{p^{1+1/\ell}}{DN^2}\)^{1/(2\ell)}  p^{o(1)}.
$$
\end{lemma}

\begin{proof} 
We follow the proof of~\cite[Theorem~4.3]{KMS2}, which is conveniently summarised in~\cite[Section~4.3]{BaSh} and also extended to
sums when the non-smooth variable runs through an arbitrary set.  Let
$$ S  =  \sum_{d\le D} \sum_{n\in \cN_d} \alpha_d \cK_{s,p}(d^r n).  $$
Let $A$ and $B$ be integer parameters for which 
\begin{equation}
\label{eq:AB N}
2AB\le N
\end{equation}
to be chosen later. By introducing averages over $a \sim A$ and $b \sim B$ (where $a \sim A$ denotes the dyadic range $A \leq a < 2A$ and similarly for $b \sim B$) and replacing $n$ by $n+ab$, we have
\begin{equation*}
\begin{split}
S &= \frac{1}{AB} \sum_{a \sim A} \sum_{b \sim B} \sum_{d \leq D} \alpha_d \sum_{\substack{n\in \Z\\ n+ab \in \cN_d}} \cK_{s,p}(d^r(n+ab))  \\
&= \frac{1}{AB} \sum_{a \sim A} \sum_{d \leq D}  \alpha_d \sum_{n\in \Z}  \sum_{\substack{b \sim B \\ n+ab \in \cN_d}} \cK_{s,p}(d^r(n+ab)). 
\end{split}
\end{equation*}
Since the range for the inner sum over $b$ is an interval, by the completing technique  (see~\cite[Section~12.2]{IwKow}), we have
$$ S \ll \frac{\log p}{AB} \sum_{a \sim A} \sum_{d \leq D} \sum_{n=-N}^N \left|\sum_{b \sim B}\cK_{s,p}(d^r(n+ab)) \e(bt) \right| $$
for some $t \in \R$,  where  $ \e(z) = \exp(2\pi i z)$.
 By making a change of variables $u = d^ra$ and $v = \overline{a}n$, we obtain
$$ S \ll \frac{\log p}{AB} \sum_{u, v \in \F_p} \nu(u,v) \left|\sum_{b \sim B} \cK_{s,p}(u(v+b)) \e(bt) \right|, $$
where $\nu(u,v)$ is the number of triples $(a,m,n)$ with $a \sim A$,  $m \in  \{d^r:~1\leq d \leq D\}$ and $n \in [-N, N]$ such that
$$ u \equiv \overline{a}n \pmod{p} \mand v \equiv ma\pmod{p}. $$

Following the steps in~\cite[Section~4.3]{BaSh} leading to~\cite[Equations~(4.9), (4.10), and~(4.11)]{BaSh}
and defining $\cM =   \{d^r:~1\leq d \leq D\}$,
we obtain
\begin{equation}
\label{eq:FKM-bound}
S^{2\ell}
\ll A^{-2}B^{-2\ell}  D^{2\ell-2} N^{2\ell-1}  \(B^\ell p^2 + B^{2\ell} p\) J(2A,\cM) p^{o(1)},  
\end{equation}
where  $J(H,\cM)$ denotes the number of solutions to the congruence
\begin{equation}
\label{eq:Congr}
x k\equiv ym\bmod p
\end{equation}
for which $x,y\in[1,H]$ and $k,m \in \cM$. 

Taking $B = \lfloor p^{1/\ell}\rfloor$, we see that~\eqref{eq:FKM-bound} simplifies as 
 \begin{equation}
\label{eq:FKM-bound-simpl}
S^{2\ell} \ll A^{-2} D^{2\ell-2} N^{2\ell-1} J(2A,\cM) p^{1+o(1)}. 
\end{equation}

In the setting of the proof of~\cite[Theorem~4.3]{KMS2}, 
the condition 
 \begin{equation}
\label{eq:FKM-cond}
2A\max_{m\in \cM} m < p
\end{equation}
was satisfied, which allows us to replace~\eqref{eq:Congr}
with an equation $x k = ym$ in integers. Then, using the classical 
bound on the divisor function (see~\cite[Equation~(1.81)]{IwKow}),
it is shown in~\cite{KMS2} that, under the condition~\eqref{eq:FKM-cond}
we have $J(2A,\cM) \le \(A\# \cM\)^{1+ o(1)}$. However,~\eqref{eq:FKM-cond} 
is too restrictive for our purpose, so we instead use a result of 
Ayyad, Cochrane and Zheng~\cite[Theorem~2]{ACZ} 
which,  similarly as in the proof of~\cite[Theorem~4.1]{CSZ},  leads to the bound 
\begin{equation}
\label{eq:J Bound}
J(2A,\cM)  \ll A^2D^2/p + (AD)^{1+ o(1)}. 
\end{equation}
Substituting~\eqref{eq:J Bound} in~\eqref{eq:FKM-bound-simpl}, we obtain 
$$
S^{2\ell} \ll  \(D^{2\ell} N^{2\ell-1}  + A^{-1} D^{2\ell-1} N^{2\ell-1}p\)  p^{o(1)}.
$$

Since $N > 2p^{1/\ell} \ge 2B$, we may choose
$$
A=\Big\lfloor \frac{N}{2B}\Big\rfloor \gg N p^{-1/\ell}
$$
which guarantees that the condition~\eqref{eq:AB N} is met. We obtain the stated bound
after simple calculations.
\end{proof} 

\section{Proofs of Main Results}

\subsection{Proof of Theorem~\ref{thm:SquareFree any s}}\label{sec:SquareFree}

Using inclusion-exclusion, we can write 
$$
Q_{s, p}(N) =
\sum_{d \le N^{1/2} }\mu(d)\sum_{n\le N/d^2} \cK_{s, p} (d^2n).
$$
Next, we  split the sum $Q_{s, p}(N)$
 into dyadic intervals  with respect to some parameter $D \ge 1$ to get $O(\log N)$ sums of the type
$$
 S(D,N)=\sum_{d\sim D}\mu(d)\sum_{n\le N/d^2} \cK_{s, p} (d^2n)
$$
for some $D \le N^{1/2}$. By Lemma~\ref{lem:FKM Incompl} we have 
\begin{equation} \label{eq: T-PV}
    S(D,N)\ll  Dp^{1/2}\log p.
\end{equation}
By the Deligne bound~\eqref{eq:Deligne} we have
\begin{equation} \label{eq: T-Deligne}
S(D,N) \ll DN/D^2 = N/D. 
\end{equation}

On the other hand, for any fixed  even  integer $\ell>0$, by Lemma~\ref{lem: Type I any s} we have
$$
S(D,N) \ll  DK\(K^{-1}  +  \frac{p^{1+1/\ell}}{DK^2}\)^{1/(2\ell)}  p^{o(1)},
$$
where $K=N/D^2$, provided that $K > 2p^{1/\ell}$. Note that
$$ DK =   N/D \le N \le p \le  p^{1+1/\ell}, $$
and thus
$$
K^{-1} \le \frac{p^{1+1/\ell}}{DK^2}. 
$$
It follows that if $K > 2p^{1/\ell}$ then
\begin{equation}    \label{eq:BilinBound s}
\begin{split}
S(D,N) &\ll DK\(\frac{p^{1+1/\ell}}{DK^2}\)^{1/(2\ell)}p^{o(1)} \\
&=  \frac{N}{D}\left( \frac{D^3p^{1+1/\ell}}{N^2} \right)^{1/(2\ell)}p^{o(1)}.
\end{split}
\end{equation}

Using the bounds~\eqref{eq: T-PV} or~\eqref{eq:BilinBound s} for $K > 2p^{1/\ell}$ and the bound~\eqref{eq: T-Deligne} for $K \leq 2p^{1/\ell}$, we arrive at
\begin{equation}\label{eq: T-overall}
S(D,N) \ll  \min\left\{f_1(D), f_2(D)\right\} p^{o(1)} + N^{1/2} p^{1/2\ell},
\end{equation}
where 
$$ 
f_1(D) = Dp^{1/2} \mand  f_2(D) = \frac{N}{D}\left( \frac{D^3p^{1+1/\ell}}{N^2} \right)^{1/(2\ell)}.
$$
Choose parameter 
$$ D_0 = \(\frac{N^{2\ell-2}}{p^{\ell-1-1/\ell}}\)^{\tfrac{1}{4\ell-3}} 
= N^{1/2}p^{-1/4} \(\frac{p^{1/2+2/\ell}}{N}\)^{\tfrac{1}{2(4\ell-3)}}  $$ 
such that $f_1(D_0) = f_2(D_0)$. 
Clearly $1 \le D_0 \le N^{1/2}$ since $N \geq p^{1/2 + 2/\ell}$ by assumption.
Hence
$$
\min\left\{f_1(D), f_2(D)\right\}   \le f_1(D_0)   \le N^{1/2}p^{1/4} \(\frac{p^{1/2+2/\ell}}{N}\)^{\tfrac{1}{2(4\ell-3)}}.
$$
It can be easily verified that
$$ p^{\frac{1}{2\ell}} \le p^{1/4} \(\frac{p^{1/2+2/\ell}}{p}\)^{\tfrac{1}{2(4\ell-3)}} \le p^{1/4} \(\frac{p^{1/2+2/\ell}}{N}\)^{\tfrac{1}{2(4\ell-3)}}. $$
Hence the second term on the right-hand side of~\eqref{eq: T-overall} is dominated by the first term and we have
$$
S(D,N) \ll  N^{1/2}p^{1/4} \(\frac{p^{1/2+2/\ell}}{N}\)^{\tfrac{1}{2(4\ell-3)}} p^{o(1)}.
$$
This concludes the proof.

\subsection{Proof of Theorem~\ref{thm:smooth-Kloosterman}}\label{sec:smooth}

Let $L_0 \in [1, N]$ be a parameter to be chosen later. Observe that any $y$-smooth integer in $(L_0, N]$ can be uniquely factored as $n = \ell m$ such that
$$ \ell \in (L_0, yL_0], \qquad \frac{\ell}{P(\ell)} \leq L_0,\qquad p(m) \geq P(\ell), $$
where $P(\ell)$ denotes the largest prime factor of $\ell$ and $p(m)$ denotes the smallest prime factor of $m$. Indeed, this factorization can be obtained by writing $n = p_1p_2\cdots p_k$ with primes $p_1 \le \cdots \le p_k$ and by setting $\ell = p_1p_2\cdots p_r$ where $r$ is the smallest positive integer such that $p_1p_2\cdots p_r > L_0$.

 Thus  we have
$$R_{s,p}(N,y) 
= \sum_{\substack{L_0 < \ell \leq P(\ell) L_0 \\ \ell \in S(y)}} \, \sum_{\substack{m \in S(N/\ell, y) \\ p(m) \geq P(\ell)}} \cK_{s,p}(\ell m) + O(L_0). $$
After dyadic partition of the range for $\ell$, we see that there is $L \in (L_0, yL_0]$ such that 
\begin{equation}
\label{eq:  R and U}
R_{s,p}(N,y)  \ll U \log N + L_0, 
\end{equation}
where 
$$ U= \sum_{\substack{L < \ell \leq \min\left\{P(\ell)L_0, 2L\right\} \\ \ell \in S(y)}} \,  \sum_{\substack{m \in S(N/\ell, y) \\ p(m) \geq P(\ell)}} \cK_{s,p}(\ell m). 
$$

We now employ  the completing technique as in~\cite[Section~12.2]{IwKow} again. 
That is, first  we write 
\begin{align*}
 U& =  \sum_{\substack{L < \ell \leq \min\left\{P(\ell)L_0, 2L\right\} \\ \ell \in S(y)}} \,  \sum_{\substack{m \in S(N/L, y) \\ p(m) \geq P(\ell)}} \cK_{s,p}(\ell m) \\
 & \qquad \qquad \qquad\qquad\qquad  \frac{1}{N} \sum_{ 1 \le k \le N/\ell} \sum_{a=1}^N   \e\(a (m-k)/N\).
\end{align*}
After changing the order of summation and using~\cite[Equation~(8.6)]{IwKow} (similarly to how we have 
done  in Section~\ref{sec:SquareFree}), we derive 
 $$ U \ll  \sum_{\substack{\ell \sim L \\ \ell \in S(y)}} \left| \sum_{\substack{m \in S(N/L, y) \\ p(m) \geq P(\ell)}} \cK_{s,p}(\ell m) \e(\eta m) \right| \log N
$$
for some real $\eta \in \R$,  where, as before,  $ \e(z) = \exp(2\pi i z)$.

By the Cauchy--Schwarz inequality,  we have
$$ U^2 \ll (\log N)^2 \Psi(L, y)  \sum_{\substack{\ell \sim L \\ \ell \in S(y)}} \left| \sum_{\substack{m \in S(N/L, y) \\ p(m) \geq P(\ell)}} \cK_{s,p}(\ell m) \e(\eta m) \right|^2. $$
Writing $q = P(\ell)$ and replacing $\ell$ by $q\ell$, we have
$$U^2 \ll (\log N)^2 \Psi(L, y)  \sum_{q \leq y} \sum_{\ell \sim L/q} \left| \sum_{\substack{m \in S(N/L, y) \\ p(m) \geq q}} \cK_{s,p}(q \ell m) \e(\eta m) \right|^2, $$
where we have dropped the primality condition on $q$. Expanding the square, we get
\begin{equation}\label{eq:smooth1} 
\begin{split}
U^2 \ll  & (\log N)^2 \Psi(L,y) \\ 
&\times \sum_{q \leq y} \sum_{m_1,m_2 \in S(N/L,y)} \left| \sum_{\ell \sim L/q} \cK_{s,p}(q\ell m_1) \overline{\cK_{s,p}(q\ell m_2)} \right|. 
\end{split} 
\end{equation}

The contribution $Y_1$ to the right-hand side of~\eqref{eq:smooth1}  from  the diagonal terms with $m_1=m_2$ is 
$$
Y_1 \ll (\log N)^2 \Psi(L,y) \sum_{q \leq y} \Psi(N/L, y) \frac{L}{q} \ll L N^{-\alpha} \Psi(N,y)^2 (\log N)^3, $$
where $\alpha= \alpha(N,y)$ and  we have also used the standard bounds (see~\cite[Section~2]{Harp})
$$ \Psi(L, y) \ll \left(\frac{L}{N}\right)^{\alpha}\Psi(N,y),\quad\Psi(N/L,y) \ll \frac{1}{L^{\alpha}} \Psi(N,y). $$

To estimate the contribution $Y_2$ to the right-hand side of~\eqref{eq:smooth1}  from  the non-diagonal terms 
with $m_1 \neq m_2$,  we observe that Lemma~\ref{lem:FKM Incompl} implies that
$$ \left| \sum_{\ell \sim L/q} K_p(q\ell m_1) \overline{K_p(q\ell m_2)} \right| \ll p^{1/2} \log p $$
when $m_1 \neq m_2$. Hence
\begin{align*}
Y_2 & \ll (\log N)^2 \Psi(L,y)\cdot  y \Psi(N/L, y)^2 p^{1/2}\log p \\
& \ll y p^{1/2} L^{-\alpha} \Psi(N,y)^2 N^{o(1)}. 
\end{align*}

Overall, with $\alpha = \alpha(N,y)$, we then have
\begin{align*}U^2& \ll Y_1 + Y_2   \le (LN^{-\alpha} + yp^{1/2}L^{-\alpha}) \Psi(N,y)^2 N^{o(1)} \\
& \le  y\(L_0N^{-\alpha} + p^{1/2}L_0^{-\alpha}\) \Psi(N,y)^2 N^{o(1)}. 
\end{align*}
Choosing $L_0 = p^{1/2(1+\alpha)} N^{\alpha/(1+\alpha)}$ to balance the two terms depending on $L_0$, we conclude that
$$U \le y^{1/2} p^{\frac{1}{4(1+\alpha)}} N^{-\frac{\alpha^2}{2(1+\alpha)}} \Psi(N,y) N^{o(1)}. $$
Hence we see from~\eqref{eq:  R and U} that 
\begin{align*}
R_{s,p}(N,y)  & \ll  \Psi(N,y) y^{1/2} p^{\frac{1}{4(1+\alpha)}} N^{-\frac{\alpha^2}{2(1+\alpha)}+o(1)}
+ p^{1/2(1+\alpha)} N^{\alpha/(1+\alpha)}\\
& =    \Psi(N,y) y^{1/2} p^{\beta} N^{-\gamma + o(1)} +    p^{2\beta} N^{\alpha -2 \gamma}. 
\end{align*}
We may assume that $p^{\beta} \le N^{\gamma}$ since otherwise the claimed bound is trivial. Using~\eqref{eq: Psi N-alpha},    we derive
$$
 p^{2\beta} N^{\alpha -2 \gamma} \le  p^{\beta} N^{\alpha -\gamma} = \Psi(N,y)p^{\beta} N^{-\gamma + o(1)}, 
$$
and the result follows.

\end{document}